\documentclass[final,a4paper,reqno,11pt]{amsart}
\pdfoutput=1

\usepackage[final]{graphicx}
\usepackage[final]{hyperref}
\usepackage{amssymb, amstext, amsmath, amsthm, amsfonts, mathrsfs, enumerate, lmodern, ifdraft, stmaryrd, enumitem}
\usepackage[noadjust]{cite}
\usepackage[utf8]{inputenc}
\usepackage[english]{babel}
\usepackage[noabbrev]{cleveref}
\usepackage{multirow}
\usepackage{xspace}
\usepackage{float}
\usepackage{mathtools}
\usepackage{xcolor}

\newtheorem{theorem}{Theorem}[section]
\newtheorem{corollary}[theorem]{Corollary}
\newtheorem{lemma}[theorem]{Lemma}
\newtheorem{proposition}[theorem]{Proposition}

\newtheorem*{theorem*}{Theorem}
\newtheorem*{corollary*}{Corollary}
\newtheorem{introtheorem}{Theorem}

\newtheorem{introcorollary}[introtheorem]{Corollary}

\Crefname{introcorollary}{corollary}{corollary}

\theoremstyle{definition}
\newtheorem{definition}[theorem]{Definition}

\newtheorem*{conjecture*}{Conjecture}
\newtheorem*{example*}{Example}

\theoremstyle{remark}
\newtheorem{remark}[theorem]{Remark}
\newtheorem*{acknowledgment}{Acknowledgment}

\newcommand{\cC}{\mathcal{C}}

\newcommand{\sD}{\mathsf{D}}
\newcommand{\sK}{\mathsf{K}}
\newcommand{\sM}{\mathsf{M}}

\newcommand{\kk}{\Bbbk}
\newcommand{\NN}{\mathbb{N}}
\newcommand{\QQ}{\mathbb{Q}}

\newcommand{\ZZ}{\mathbb{Z}}

\newcommand{\add}{\mathsf{add}\hspace{.01in}}

\newcommand{\ind}{\mathsf{ind}\hspace{.01in}}

\newcommand{\inj}{\mathsf{inj}\hspace{.01in}}

\renewcommand{\mod}{\mathsf{mod}\hspace{.01in}}

\newcommand{\proj}{\mathsf{proj}\hspace{.01in}}
\newcommand{\simp}{\mathsf{simp}\hspace{.01in}}

\newcommand{\cok}{\operatorname{Cok}\nolimits}
\newcommand{\End}{\operatorname{End}\nolimits}
\newcommand{\Ext}{\operatorname{Ext}\nolimits}

\newcommand{\Hm}{\operatorname{H}\nolimits}
\newcommand{\Hom}{\operatorname{Hom}\nolimits}

\newcommand{\Id}{\operatorname{Id}\nolimits}
\newcommand{\im}{\operatorname{Im}\nolimits}
\renewcommand{\ker}{\operatorname{Ker}\nolimits}
\newcommand{\op}{{\operatorname{op}\nolimits}}

\newcommand{\rad}{\operatorname{rad}\nolimits}

\newcommand{\soc}{\operatorname{soc}\nolimits}
\renewcommand{\top}{\operatorname{top}\nolimits}

\renewcommand{\subset}{\subseteq}
\newcommand{\trans}{\mkern-1.5mu \mathsf{T}}

\newcommand{\includepdf}[2][]{
   \begin{center}
      \includegraphics[#1]{#2.pdf}
   \end{center}
}

\newcommand{\tm}{$\tau$-map\xspace}
\newcommand{\tim}{$\tau^{\scriptscriptstyle{-1}}$-map\xspace}
\newcommand{\s}{\mbox{}}

\begin{document}
   \title[When does the AR translation operate linearly on $\sK_0(\mod A)$? -- Part I]{When does the Auslander--Reiten translation operate linearly on the Grothendieck group? -- Part I}

   \author{Carlo Klapproth}\address{Department of Mathematics, Aarhus University, 8000 Aarhus C, Denmark} 
   \email{carlo.klapproth@math.au.dk}
    
   
   \maketitle

   \begin{abstract}
      For a hereditary, finite-dimensional algebra $A$ the Coxeter transformation extends the action of the Auslander--Reiten translation on the non-projective indecomposable modules to a linear endomorphism of the Grothendieck group of the category of finitely generated $A$-modules.
      It is natural to ask whether other algebras admit a similar linear extension. 
      We show that this is indeed the case for all Nakayama algebras.
      Conversely, we show that finite-dimensional algebras with non-acyclic and connected quiver admitting such a linear extension are already cyclic Nakayama algebras.
   \end{abstract}

   \section{Introduction}
   \noindent Let $A$ be a basic, finite-dimensional $\kk$-algebra over a field $\kk$.
   If $\kk$ is algebraically closed and $A$ has finite global dimension, then the bounded derived category $\sD^b(\mod A)$ of $A$ has Auslander--Reiten triangles by \cite[Theorem 4.6]{happel_triangulated_1988}, and hence a Serre functor $\mathbb{S}$ by \cite[Theorem I.2.4]{reiten_noetherian_2002}.
   It is well-known that $\tau \coloneqq \mathbb{S}[-1]$ is then a triangulated functor which induces the \emph{Coxeter transformation} $\Phi$, that is a linear endomorphism of the Grothendieck group $\sK_0(\sD^b(\mod A)) \cong \sK_0(\mod A)$, see e.g.\ \cite[Section III.1.2]{happel_triangulated_1988}, with $\Phi [X] = [\tau X]$ for $X \in \sD^b(\mod A)$. 

   The situation is more difficult for the category $\mod A$ of finitely generated $A$-modules.
   The Auslander--Reiten translation $\tau_{A}$ on $\mod A$ does not need to be a functor.
   However, inspired by the situation for hereditary algebras one can define the following.

   \begin{definition}
      An endomorphism $\Phi \in \End_{\ZZ}(\sK_0(\mod A))$ with $\Phi[M] = [\tau_A M]$ for all non-projective indecomposable $A$-modules $M$ is called \emph{\tm for $A$}.
      Similarly, an endomorphism $\Phi' \in \End_{\ZZ}(\sK_0(\mod A))$ with $\Phi'[M] = [\tau_A^{-1} M]$ for all non-injective indecomposable $A$-modules $M$ is called \emph{\tim for $A$}. \label{item:invtaumap}
   \end{definition}

   \noindent It is natural to ask which algebras do admit a \tm.
   Recall the following.

   \begin{definition}
      A finite-dimensional $\kk$-algebra $A$ is called \emph{Nakayama algebra} if the indecomposable projective and indecomposable injective $A$-modules are uniserial.
      A Nakayama algebra is called \emph{cyclic}, if it has no projective and no injective simple modules.
   \end{definition}

   \noindent
   We have the following examples of algebras admitting a \tm.
   \begin{enumerate}
      \item If $Q$ is an acyclic quiver and $\kk Q$ is its (hereditary) quiver algebra then $\kk Q$ has a \tm given by the Coxeter transformation, see \cite[Proposition VIII.2.2(b)]{auslander_representation_1995}. 
      \item For a Nakayama algebra $A$ there is a unique $\Phi \in \End_{\ZZ}(\sK_0(\mod A))$ with 
      \[
         \Phi [S] = \begin{cases} 
                        [\tau_A S] & \text{for $S \in \mod A$ simple and non-projective, } \\
                        x_S      & \text{for $S \in \mod A$ simple and projective,} 
                    \end{cases}
      \]
      where $x_S$ is an arbitrary, fixed element in $\sK_0(\mod A)$ for each simple projective $A$-module $S$.
      This map $\Phi$ is a \tm by \Cref{proposition:nakayamamap}. 
   \end{enumerate}

   Notice, a \tm can look very different from the Coxeter transformation.
   For example the algebra $A \coloneqq \kk Q/I$, where $Q$ is the quiver
   \includepdf{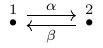}
   and $I \coloneqq \langle \alpha \beta \rangle$, is a Nakayama algebra without simple projective modules and hence has a unique \tm $\Phi$, by \Cref{proposition:nakayamamap}.
   This map transposes the two simple $A$-modules.
   On the other hand, the Cartan matrix $C_{}$ of $A$ is symmetric and hence the Coxeter matrix $C_{}^t C_{}^{-1}$ of $A$ is the identity matrix.
   So, the Coxeter transformation of $A$ is the identity map on $\sK_0(\mod A)$.

   We suspect that the class of algebras which have a \tm admits a classification. 
   Because of very different behaviour with regards to the existence of a \tm, we want to treat algebras with an acyclic $\Ext^1$-quiver (see \Cref{definition:ext-quiver}) separately from those with an acyclic $\Ext^1$-quiver.
   This paper is mainly concerned with algebras which admit a \tm and have a non-acyclic $\Ext^1$-quiver.
   We show the following, which is our main result.

   \begin{introtheorem}[\Cref{theorem:conclusion}]\label{introthm}
      Suppose the $\Ext^1$-quiver of $A$ is con\-nected and non-acyclic. 
      Then $A$ has a \tm if and only if $A$ is a cyclic Nakayama algebra.
   \end{introtheorem}

   \noindent In the language of quotients of path algebras \Cref{introthm} can be reformulated as follows.

   \begin{introcorollary}[\Cref{corollary:conclusion}]
      Let $\kk$ be an algebraically closed field.
      Suppose $A \coloneqq \kk Q / I$, where $Q$ is a connected and non-acyclic quiver and $I \lhd \kk Q$ is an admissible ideal. 
      Then $A$ has a \tm if and only if $A$ is a cyclic Nakayama algebra.
   \end{introcorollary}

   \noindent
   This shows that for an algebra with a non-acyclic $\Ext^1$-quiver the existence of a \tm restricts the shape of the underlying quiver significantly.
   This is not the case for algebras with an acyclic $\Ext^1$-quiver as indeed all hereditary algebras admit a \tm.

   \section{Preliminaries and Notation}
   \noindent For morphisms $f \colon X \to X'$ and $g \colon X' \to X''$ we write $gf$ for the composite $g \circ f \colon X \to X''$.
   Similarly, for functors $F \colon \cC \to \cC'$ and $G \colon \cC' \to \cC''$ we write $GF$ for the composite functor $G \circ F \colon \cC \to \cC''$.
   For a category $\cC$ we denote by $\ind \cC$ the full subcategory of indecomposable objects in $\cC$.
   We fix a field $\kk$, which is not necessarily algebraically closed.
   
   For arbitrary finite-dimensional $\kk$-algebras $A$ and $B$ we refer to right \mbox{$A$-modules} as \mbox{$A$-modules} and to right $A^\op \otimes_\kk B$-modules as $A$-$B$-bimodules.
   We denote by $\mod A$ the category of finitely generated $A$-modules.
   We write $\proj A$ and $\inj A$ for the full subcategories of $\mod A$ consisting of projective and injective $A$-modules, respectively.
   For $P \in \proj A$ we denote by $\simp P$ the isomorphism classes of simple quotients of $P$.
   Abusing notation, we view the elements of $\simp P$ as simple $A$-modules by picking a representative.
   We denote the $\kk$-duality $\Hom_{\kk}(-,\kk)$ by $\sD$ and write $\nu_{A}(-) \coloneqq \sD \Hom_{A}(-,A) \colon \mod A \to \mod A$ for the Nakayama functor, $\nu_{\smash{A}}^{\smash{-1}}(-) \coloneqq \Hom_{\smash{A^{\op}}}(\sD(-), A) \colon \mod A \to \mod A$ for the inverse Nakayama functor as well as $\tau_{A}$ and $\tau_A^{-1}$ for the Auslander--Reiten translations on $\mod A$.
   We want to point out that the restriction $\nu_{A} \colon \proj A \to \inj A$ is inverse to the restriction $\nu_{\smash{A}}^{-1} \colon \inj A \to \proj A$.
   However, the unrestricted functors $\nu_{A}$ and $\nu_{A}^{\smash{-1}}$ are generally not mutually inverse equivalences on $\mod A$.
   We refer the reader to \cite[Chapter IV]{auslander_representation_1995} for a more detailed exposition.

   For a simple module $S \in \mod A$ and an arbitrary module $M \in \mod A$ we denote by $[ M:S ]$ how often $S$ appears as a composition factor in any composition series of $M$.
   By the Jordan--Hölder theorem this is independent of the chosen composition series. 
   Recall that $[M:S] \neq 0$ if and only if $\Hom_{A}(P, M) \neq 0$ where $P$ is a projective cover of $S$, see for example \cite[Exercise II.1(a)]{auslander_representation_1995}.
   Compare also \cite[Corollary III.3.6]{assem_elements_2006} for the case where $\kk$ is algebraically closed.
   The following result is folklore, see e.g.\ \cite[Lemma 1]{iwanaga}.

   \begin{remark}\label{remark:extwithsimple}
      Let $\cdots \to P_1 \to P_0 \eqqcolon P_\bullet$ be a minimal projective resolution of some \mbox{$A$-module} $M$ and $S$ be a simple $A$-module.
      Applying $\Hom_A(-, S)$ to $P_\bullet$ yields a complex $\Hom_A(P_\bullet, S)$ with zero differentials since $\Hom_A(-,S)$ vanishes on all morphisms \mbox{$f \in \Hom_A(N,N')$} satisfying $\im f \subset \rad N'$ as $\rad N'$ is contained in all maximal submodules of $N'$. 
      Hence, $\Ext_A^i(M,S) \cong \Hm^i\Hom_A(P_\bullet, S) \cong \Hom_A(P_i,S)$ for $i \in \NN$.
   \end{remark}

   Throughout this paper we assume that $A$ is a basic, finite-dimensional $\kk$-algebra and that $e \in A$ is an idempotent.
   We write $\Gamma_e \coloneqq (1-e)A(1-e)$ for the idempotent subalgebra obtained from $A$ by deleting $e$.
   By projectivisation, c.f.\ \cite[Proposition II.2.1]{auslander_representation_1995}, the functor $\Hom_{A}((1-e)A,-) \colon \add (1-e)A \to \proj \Gamma_e$ is an equivalence  and maps the basic object $(1-e)A$ to $\Gamma_e$, showing that $\Gamma_e$ is also a basic finite-dimensional $\kk$-algebra.
   Recall that the simple $\Gamma_e$-modules are in correspondence with the modules in $\simp (1-e)A$, see \Cref{lemma:sturdy}(\ref{item:simple2}).
   We define the two functors 
   \begin{align*}
      F_{e}(-) &\coloneqq - \otimes_A A(1-{e}) \colon \mod A \to \mod \Gamma_{e} \text{ and} \\
      G_{e}(-) &\coloneqq - \otimes_{\Gamma_{e}} (1-{e})A \colon \mod \Gamma_{e} \to \mod A \text{.}
   \end{align*}
   Notice that $A(1-e)$ is a projective $A^{\op}$-module. 
   Hence, $- \otimes_A A(1-e) \cong \Hom_A((1-e)A, -)$, see \Cref{lemma:sturdy}(\ref{item:ashom}).
   This means that $F_e$ is right adjoint to $G_e$ by Tensor-$\Hom$ adjunction.

   If $\kk$ is algebraically closed then it is well-known that we have an isomorphism $A \cong \kk Q/I$ where $Q$ is a quiver and $I \subset \kk Q$ is an admissible ideal.
   However, as we do normally not assume $\kk$ to be algebraically closed, we define the following as a replacement.
   \begin{definition} \label{definition:ext-quiver}
      The \emph{$\Ext^1$-quiver} $Q = (Q_0, Q_1)$ of $A$ is the quiver with
      \begin{enumerate}
         \item vertices $Q_0 = \simp A$ and
         \item arrows $Q_1 = \{\,S \to S' \,|\, \text{$S,S' \in Q_0$ and $\Ext_A^1(S,S') \neq 0$ }\}$.
      \end{enumerate}
   \end{definition} 
   \noindent Notice, the $\Ext^1$-quiver of $A$ can have $1$-cycles, that is loops.
   Furthermore, it has \emph{precisely one} arrow from $S$ to $S'$ if $S'$ has non-trivial extensions by $S$ and \emph{no arrow} from $S$ to $S'$ otherwise.
   For an admissible quotient of a path algebra the $\Ext^1$-quiver relates to the underlying quiver of the path algebra as follows. 

   \begin{remark} \label{remark:quivertoextquiver}
      Suppose $\kk$ is algebraically closed. 
      Let $Q$ be a quiver and $I \lhd \kk Q$ be an admissible ideal.
      Then by \cite[Proposition III.2.12(a)]{assem_elements_2006} the $\Ext^1$-quiver of $\kk Q/I$ is obtained from $Q$ by replacing all vertices with their corresponding simple $\kk Q /I$-modules and by identifying parallel arrows.
      In particular, acyclicity and connectedness hold for $Q$ if and only if they hold for the $\Ext^1$-quiver of $\kk Q/I$.
   \end{remark}

   \noindent Recall that a vertex of a quiver is called \emph{source vertex} if there are no arrows ending in it.
   Dually a vertex is called \emph{sink vertex} if there is no arrow starting in it.
   Notice, this does not imply that there is any arrow starting in a source vertex or ending in a sink vertex.
   Indeed, a vertex is \emph{isolated} if and only if it is both a source and a sink vertex. 
   
   For $n=1$ we define $C_n$ to be a single vertex with a loop attached to it.
   For $n \geq 2$ we denote the oriented cycle on $n$ vertices
   \includepdf{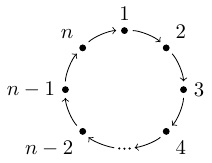}
   by $C_n$.
   For quivers $Q'$ and $Q''$ we denote the disjoint union of $Q'$ and $Q''$ by $Q' \sqcup Q''$.
   If $Q = Q' \sqcup Q''$ and $Q'$ is connected, then we call $Q'$ a \emph{component} of $Q$.

   For a basic algebra $A$ and a decomposition $1 = e_1 + \dots + e_n$ into primitive idempotents, the $A$-modules $e_i A/\rad e_i A$ for $1 \leq i \leq n$ form a full system of representatives of $\simp A$.
   
  


   For convenience, we recall some well-known facts about $F_e$ and $G_e$. 
   These functors have been studied for example in \cite[Section 5 and 6]{auslander}, \cite[Section 6.2]{green} and \cite{psaroudakis_homological_2014}.

   \begin{lemma} \label{lemma:sturdy}
      The following statements hold.
      \begin{enumerate}
         \item There is an equivalence $\add (1-e)A \leftrightarrow \proj \Gamma_e$ induced by $F_e$ and $G_e$. \label{item:equiproj}
         \item The functor $F_e$ is exact. 
               If $(1-e)Ae \in \proj \Gamma_e^{\op}$ then $G_e$ is exact. \label{item:exact}
         \item The functor $F_e$ is left inverse to $G_e$ that is $F_e G_e \cong \Id_{\mod \Gamma_e}$. \label{item:inverse}
         \item We have $F_e(-) = \Hom_A((1-e)A,-)$.\label{item:ashom}
         \item The functor $F_e$ is right adjoint to $G_e$. \label{item:adjointness}
         \item The functor $G_e$ preserves augmented (minimal) projective presentations. 
               Hence, $G_e$ preserves the properties of being non-projective and being indecomposable.\label{item:projpres}
      \end{enumerate}
      Let $1-e = e_1 + \cdots + e_n$ be a decomposition into primitive idempotents in ${\Gamma_e \subset A}$.
      \begin{enumerate}[resume]
         \item We have $F_e(e_i A / \rad e_i A) \cong e_i \Gamma_e / \rad e_i \Gamma_e$ for $1 \leq i \leq n$. 
            Hence, $F_e$ induces a bijection between $\simp (1-e)A$ and $\simp \Gamma_e$.\label{item:simple2}
         \item If $(1-e)Ae = 0$ then $G_e(e_i \Gamma_e/ \rad e_i \Gamma_e) \cong e_i A / \rad e_i A$ for $1 \leq i \leq n$.
            Hence, $G_e$ induces a bijection between $\simp \Gamma_e$ and $\simp (1-e)A$.
            \label{item:simple1}
         \item If $(1-e)Ae = 0$ then $\Ext^1_{\Gamma_e}(S,S') \cong \Ext^1_{A}(G_e(S),G_e(S'))$ for $S,S' \in \simp \Gamma_e$.\label{item:sturdyext}
      \end{enumerate}
   \end{lemma}
   \begin{proof}
      Item (\ref{item:equiproj}) follows from projectivisation, cf.\ \cite[Proposition II.2.1]{auslander_representation_1995}.
      Item (\ref{item:exact}) is a consequence of $A(1-e) \in \proj A^{\op}$ and $(1-e)A \cong \Gamma_e \oplus (1-e)Ae$ as $\Gamma_e$-modules, implying that $(1-e)A \in \proj \Gamma_e$ if and only if $(1-e)Ae \in \proj \Gamma_e$.
      Item (\ref{item:inverse}) follows, because $(1-e) A \otimes_A A(1-e) \cong \Gamma_e$ holds as $\Gamma_e$-$\Gamma_e$-bimodules.
      Item (\ref{item:ashom}) follows from 
      \[F_e(-) = - \otimes_A A(1-e) \cong - \otimes_A \Hom_A((1-e)A, A) \cong \Hom_A((1-e)A, -)\]
      because $A(1-e) \in \proj A^{\op}$, see \cite[Proposition II.4.4]{auslander_representation_1995} for the last equation.
      Then (\ref{item:adjointness}) follows from (\ref{item:ashom}) using Tensor-$\Hom$ adjunction. 

      To show (\ref{item:projpres}) notice first that $G_e(-) = - \otimes_{\Gamma_e} (1-e)A$ is right exact.
      Hence, $G_e$ maps augmented projective presentations to augmented projective presentations using (\ref{item:equiproj}). 

      We show that $G_e$ preserves minimality. 
      A projective presentation $f \colon P_1 \to P_0$ is minimal if and only if the two term complex $f \colon P_1 \to P_0$ has no non-trivial direct summand of the form $f' \colon P_1' \to P_0'$ with $f'$ being a split epimorphism, c.f.\ \cite[Corollary 4.10]{auslander}.
      Since equivalences preserve direct summands and split epimorphisms, it follows from (\ref{item:equiproj}) that $G_e$ preserves minimal projective resolutions.

      A module with a minimal projective presentation $f \colon P_1 \to P_0$ is non-projective if and only if $P_1$ is non-zero and indecomposable if and only if $f \colon P_1 \to P_0$ is indecomposable as a $2$-term complex. 
      Both properties are preserved by $G_e$ by (\ref{item:equiproj}) and the first part of (\ref{item:projpres}).

      We show (\ref{item:simple2}).
      Let $P \to e_i A \to e_i A/\rad e_i A \to 0$ be an augmented minimal projective presentation.
      Since $F_e$ is exact there is an exact sequence 
         \begin{equation} F_e(P) \to F_e(e_i A) \to F_e(e_i A / \rad e_i A) \to 0 \label{eq:projres} \end{equation}
      in $\mod \Gamma_e$ and $F_e(e_i A) \cong e_i \Gamma_e$ is an indecomposable projective $\Gamma_e$-module.
      Furthermore, $F_e(e_i A / \rad e_i A)$ is simple by \cite[Corollary 6.3d)]{auslander} and therefore the sequence (\ref{eq:projres}) implies $F_e(e_i A / \rad e_i A) \cong \top e_i \Gamma_e = e_i \Gamma_e / \rad e_i \Gamma_e$.

      We show (\ref{item:simple1}).
      Assume $(1-e)Ae = 0$ and consider again the exact sequence arising as in (\ref{eq:projres}).
      We have $\Hom_A(eA, e_i A) \subseteq \Hom_A(eA, (1-e)A) \cong (1-e)Ae = 0$, that is every map from an object in $\add eA$ to $e_i A$ is trivial.
      Then $P \in \add (1-e) A$, by minimality of the chosen projective resolution of $e_iA/\rad e_i A$.
      Hence, $F_e(P)$ is projective and (\ref{eq:projres}) is an augmented projective presentation of $F_e(e_iA /\rad e_i A) \cong e_i \Gamma_e / \rad e_i \Gamma_e$.
      Applying the functor $G_e$ to (\ref{eq:projres}) shows that $G_eF_e(P) \to G_eF_e(e_i A)$ is a projective presentation of $G_e(e_i \Gamma_e / \rad e_i\Gamma_e)$, by (\ref{item:projpres}).
      However, as $e_iA$ and  $P$ are in $\add (1-e)A$ this projective presentation is equivalent to $P \to e_i A$ by (\ref{item:equiproj}) and hence $G_e(e_i \Gamma_e / \rad e_i \Gamma_e) \cong e_i A / \rad e_i A$.
      
      We show (9). 
      Suppose ${P_1 \to P_0 \to S \to 0}$ is an augmented minimal projective presentation of the simple $\Gamma_e$-module $S$.
      Then ${G_e(P_1) \to G_e(P_0) \to G_e(S) \to 0}$ is an augmented minimal projective presentation of $G_e(S)$ by (\ref{item:projpres}).
      We know that $G_e(S')$  is a simple $A$-module as $S'$ is a simple $\Gamma_e$-module, using (\ref{item:simple1}).
      Applying \Cref{remark:extwithsimple} we obtain $\Ext_{\Gamma_e}^1(S, S') \cong \Hom_{\Gamma_e}(P_1, S')$ and  $\smash{\Ext_A^1(G_e(S), G_e(S')) \cong \Hom_A(G_e(P_1), G_e(S'))}$.
      By (\ref{item:inverse}) and (\ref{item:adjointness}) we have $\Hom_A(G_e(P_1), G_e(S')) \cong \Hom_{\Gamma_e}(P_1, S')$.
      Finally, combining all of the above equations we obtain $\Ext_{\Gamma_e}^1(S, S') \cong \Ext_A^1(G_e(S), G_e(S'))$.
   \end{proof}

   \noindent Next, we want to recall how the $\Ext^1$-quiver of $A$ and $\Gamma_e$ are related if $(1-e)Ae = 0$.

   \begin{lemma} \label{lemma:vertexremoval}
      Let $Q_A$ be the $\Ext^1$-quiver of $A$ and $Q_{\Gamma_e}$ be the $\Ext^1$-quiver of $\Gamma_e$.
      Suppose that $(1-e)Ae = 0$. 
      Then the morphism of quivers $Q_{\Gamma_e} \to Q_A$ mapping 
      \begin{enumerate}
         \item a vertex $S$ to the vertex $G_e(S)$ and 
         \item an arrow $S \to S'$ to the arrow $G_e(S) \to G_e(S')$ 
      \end{enumerate}
      is the inclusion of the full subquiver of $Q_A$ having the elements of $\simp (1-e)A$  as vertices.
   \end{lemma}
   \begin{proof}
      By \Cref{lemma:sturdy}(\ref{item:simple1}) the functor $G_e$ induces a bijection between the vertices of $Q_{\Gamma_e}$ and the vertices of $Q_A$ which are in $\simp (1-e)A$.
      By \Cref{definition:ext-quiver} and \Cref{lemma:sturdy}(\ref{item:sturdyext}) there is an arrow $S \to S'$ in $Q_{\Gamma_e}$ if and only if there is an arrow ${G_e(S) \to G_e(S')}$ in $Q_A$.
   \end{proof}

   \begin{lemma}\label{lem:topradm}
       We have $\Ext_A^1(\top M, S) \neq 0$ for any $A$-module $M$ and any non-trivial direct summand $S$ of $\top (\rad M)$. 
   \end{lemma}
   \begin{proof}
   Consider the commutative diagram with exact rows
   \includepdf{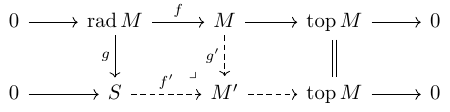}
   where $g$ is a projection onto the direct summand $S$ and the lower row is obtained by a pushout from the upper row.
   We have $\im f' = \im f' g = \im g' f = g' (\rad M) \subset \rad M'$ using that $g$ is epic.
   If the lower row was a split sequence then $M'$ would be semisimple and hence $\rad M' = 0$ implying $\im f' = 0$.
   However, $\im f' \cong S$ is non-zero, so the lower sequence defines a non-trivial element in $\Ext_A^1(\top M, S)$.
   \end{proof}
   
   \noindent The $\Ext^1$-quiver contains information about the composition series of projective modules.

   \begin{lemma} \label{lemma:extquiver}
      Let $S,S' \in \simp A$ and $P' \to S'$ be a projective cover.
      If there is $k \in \NN$ so that $S \in \add (\top (\rad^k P'))$ then there is an oriented path of length $k$ from $S'$ to $S$ in the $\Ext^1$-quiver $Q$ of $A$.
      So, all composition factors $S$ of $P'$ admit a path from $S'$ to $S$ in $Q$.
   \end{lemma}
   \begin{proof}
      If $S$ is a composition factor of $P'$ then $S$ is a direct summand of $\top (\rad^k P')$ for some $k \geq 0$, so the second part of the lemma follows from the first part.
      
      We use induction on $k \in \NN$ to prove the first part of the lemma.
      If $k =0 $ then we have $S \in \add (\top P')$ and so $S \cong S'$.
      Hence, the constant path at $S'$ gives the desired path of length $0$.
      For the induction step suppose $S$ is a simple summand of $\top (\smash{\rad^k P'})$ for some $k \geq 1$.
      Then $\Ext_A^1(\top (\rad^{k-1} P'), S) \neq 0$, by \Cref{lem:topradm}.
      Therefore, there is a simple summand $S''$ of $\top (\rad^{k-1} P_i)$ with $\Ext_A^1(S'', S) \neq 0$.
      Hence, there is an arrow from $S''$ to $S$ in $Q$ and by induction hypothesis there is an oriented path of length $k-1$ from $S'$ to $S''$ in $Q$. 
      Concatenation gives a path of length $k$ from $S'$ to $S$ in $Q$. 
   \end{proof}

   \noindent In particular, connectedness of $Q$ corresponds to indecomposability of $A$ as an algebra, and projective and injective simple modules correspond to sink and source vertices of $Q$.

   \begin{remark} \label{remark:extquiver1}
      A component $C$ of the $\Ext^1$-quiver $Q$ of $A$ induces a central idempotent $e$ and an idempotent subalgebra $\Gamma_{C} \coloneqq \Gamma_{1-e}$.
      Indeed, let $eA$ be the projective cover of the direct sum of all vertices in $C$. 
      Any vertex $S \in C$ does not appear in any composition series of $(1-e)A$, by \Cref{lemma:extquiver}.
      Therefore, $\Hom_A(eA, (1-e)A) = 0$ and similarly one can show, $\Hom_A((1-e)A, eA)= 0$. 
      Hence, $A \cong \End_A((1-e)A \oplus eA) \cong \Gamma_{e} \times \Gamma_{1-e}$ as $\kk$-algebras, showing that $e \in A$ is a central.
   \end{remark}
   
   \begin{remark} \label{remark:extquiver2}
      Suppose $S$ is a sink vertex of $Q$. 
      Let $P \to S$ be a projective cover.
      \Cref{lemma:extquiver} implies $\rad P = 0$ as it has no composition factors.
      Hence, $S = P$ is projective.
      On the other hand, a non-sink vertex in $Q$ has non-trivial extensions and is hence non-projective.
      We see that the projective simple $A$-modules are precisely the sink vertices of $Q$.
      Using $\kk$-duality we see that source vertices in $Q$ are precisely the injective simple $A$-modules.
   \end{remark}

   \noindent Using \Cref{remark:extquiver1} we can easily extend the following result to algebras $A$ where only some of the simple $A$-modules have $\tau_A$-orbits consisting of simple modules only. 

   \begin{theorem}[{\cite[Theorem IV.2.10]{auslander_representation_1995}}] \label{theorem:arsnakayama}
      Suppose all modules in the $\tau_A$-orbits of simple $A$-modules are simple themselves.
      Then $A$ is a Nakayama algebra.
   \end{theorem}

   \begin{corollary} \label{corollary:permutesummand}
      Suppose $\simp (1-e)A$ is a union of $\tau_A$-orbits.
      Then we have a decomposition $A = \Gamma_e \times \Gamma_{1-e}$ where $\Gamma_e$ is a Nakayama algebra.
   \end{corollary}
   \begin{proof}
      Let $S \in \simp (1-e)A$ and $\Ext_A^1(S, S') \neq 0$ for some simple module $S' \in \mod A$.
      There is a non-split short exact sequence $0 \to S' \to M' \to S \to 0$.
      In particular, $S \notin \proj A$ and we have an Auslander--Reiten sequence $0 \to \tau_A S \to M \to S \to 0$.
      As $M \to S$ is a right almost split morphism we obtain the dashed morphism of the diagram
      \includepdf{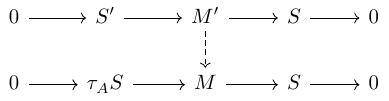}
      using that $M' \to S$ is not a split epimorphism and that $M'$ is indecomposable. 
      But because $M'$ and $M$ must be uniserial of length $2$ the dashed morphism must be an isomorphism and therefore $S' \cong \tau_A S$ holds.
      Hence, $\Ext_A^1(S,S') \neq 0$ can hold only if $S' \in \simp (1-e)A$.
      In the same way we can show that $\Ext_A^1(S'',S) \neq 0$ for $S \in \simp(1-e)A$ and $S'' \in \simp A$ implies $S'' \in \simp(1-e)A$.
      But the $\Ext^{1}$-quiver of $A$ is then disconnected and $A = \Gamma_e \times \Gamma_{1-e}$ by \Cref{remark:extquiver1} for a central idempotent $e$.
      Using for example \Cref{proposition:artranslate} one can show that $\tau_{\Gamma_e} = F_e \tau_A G_e$.
      This shows that all $\Gamma_e$-modules in the $\tau_{\Gamma_e}$-orbits of simple $\Gamma_e$-modules are simple using \Cref{lemma:sturdy}(\ref{item:simple2}) and (\ref{item:simple1}).
      Now, \Cref{theorem:arsnakayama} applied to $\Gamma_e$ shows that $\Gamma_e$ is a Nakayama algebra.
   \end{proof}

   \noindent Next, recall the following well-known extension of \cite[Proposition II.4.4(a)]{auslander_representation_1995}.

   \begin{lemma} \label{lemma:ass}
      The family of morphisms 
      \begin{alignat*}{2}
         \psi_{M,N} \colon  N \otimes_A \Hom_A(M, A)   & \to \Hom_A(M,N) \\ 
         n \otimes f   &\mapsto (m \mapsto nf(m))
         \intertext{for $M \in \mod A$ and $N \in \mod A^{\op}$, defines a natural transformation}
         \psi \colon {-}_{2} \otimes_A \Hom_A(-_{1}, A) &\to \Hom_A(-_{1}, -_{2})
      \end{alignat*}
      of functors $\mod A \times \mod A^{\op} \to \mod \kk$. Furthermore,
      \begin{enumerate}
         \item if ${}_B N_A$ is an $B$-$A$-bimodule then $\psi_{M,N}$ is also a $B^{\op}$-module morphism from \mbox{${}_B N \otimes_A \Hom_A(M, A)$} to $\Hom_A(M, {}_B N)$ for any $M \in \mod A$\label{item:Bstructure} and
         \item if $N_A \in \proj A$ then $\psi_{M,N}$ is an isomorphism for all $M \in \mod A$.\label{item:isomorphism}
      \end{enumerate}
      In particular, if the conditions of (\ref{item:Bstructure}) and (\ref{item:isomorphism}) are satisfied then there is natural isomorphism $\phi \colon {}_B N \otimes_A \Hom_A(-,A) \to \Hom_{A}(-, {}_B N)$ of functors $\mod A \to \mod B^{\op}$. 
   \end{lemma}
   \begin{proof}
      By \cite[Proposition II.4.4(a)]{auslander_representation_1995} the family of morphisms $\psi$ is a natural transformation.
      Item (\ref{item:Bstructure}) follows since $(\psi_{M,N}(bn,f))(m) = bnf(m) = (b\psi_{M,N}(n,f))(m)$
      for $m \in M$, $n \in N$, $b \in B$ and $f \in \Hom_A(M,A)$.
      
      The morphism $\psi_{M,A}$ is the canonical isomorphism between $A \otimes_A \Hom_A(M,A)$ and $\Hom_A(M,A)$, so (\ref{item:isomorphism}) holds for $N_A = A_A$.
      Since $\psi$ is natural and both $-_1\otimes_A\Hom_A(-_2,A)$ and $\Hom_A(-_2,-_1)$ are additive in $-_{1}$, this implies (\ref{item:isomorphism}) for $N_A \in \add A = \proj A$.

      For any $B$-$A$-bimodule ${}_B N_A$ which is projective as an $A$-module the natural transformation $\phi \colon {}_B N \otimes_A \Hom_A(-,A) \to \Hom_{A}(-, {}_B N)$ defined by $\phi_M \coloneqq \psi_{M,N}$ for $M \in \mod A$ is a natural isomorphism by (\ref{item:Bstructure}) and (\ref{item:isomorphism}) since a morphism of $B^{\op}$-modules is an isomorphism if and only if it is an isomorphism of $\kk$-modules.
      This shows the last part of the statement.
   \end{proof}

   \noindent Finally, recall the following from linear algebra.
   \begin{lemma} \label{lemma:Nmatrices}
      Let $X, X' \in \sM_n(\NN)$ be matrices satisfying $XX' = 1_{n\times n}$.
      Then $X$ and $X'$ are permutation matrices (cf.~e.g.~\cite[Section 1.1.2]{serre_matrices_2002}).
   \end{lemma}
   \begin{proof}
      In the ring $\sM_n(\QQ) \supset \sM_n(\NN)$ we have $X' = X^{-1}$ by \cite[Proposition 2.2.1]{serre_matrices_2002}.
      Hence, $X'X = 1_{n \times n}$ holds in $\sM_n(\NN)$, so $X$ and $X'$ define bijective maps from $\NN^n$ to $\NN^n$.
      They map $\NN^n \setminus \{0\}$ into $\NN^n \setminus \{0\}$ and hence form semigroup automorphisms of $(\NN^n \setminus \{0\}, +)$.
      The only elements of $(\NN^n \setminus \{0\}, +)$ that are not the sum of any two other elements are the standard basis vectors.
      But then $X$ and $X'$ have to permute those elements.
   \end{proof}

   \section{Core lemmas} \label{sec:core}

   \noindent In this \namecref{sec:core} we want to build the core lemmas for our classification result. 
   The main results of this section are the formula of \Cref{proposition:artranslate} showing how to calculate $\tau_{\Gamma_e}$ using $\tau_A$ as well as \Cref{lemma:fork2}.

   We want to show that \tm{s} are inherited by idempotent subalgebras arising from the removal of injective simple modules.

   \begin{lemma} \label{lemma:nakayama}
      There is a natural isomorphism $F_e \nu_A G_e \cong \nu_{\Gamma_e}$.
   \end{lemma}
   \begin{proof}
      We have $F_e(-) \cong \Hom_A((1-e)A, -)$ by \Cref{lemma:sturdy}(\ref{item:ashom}) and therefore
      \[F_e \sD(-) \cong \Hom_A((1-e)A, \Hom_\kk(-,\kk)) \cong \Hom_\kk((1-e)A\otimes_A-,\kk) = \sD((1-e)A \otimes_A-)\] 
      by Tensor-$\Hom$ adjunction.
      This, \Cref{lemma:ass} and Tensor-$\Hom$ adjunction give a sequence
      \begin{align*}
         F_e (\sD \Hom_A(- \otimes_{\Gamma_e} (1-e) A, A)) & \cong \sD ((1-e)A \otimes_A \Hom_A(- \otimes_{\Gamma_e} (1-e)A, A)) \\
                                                                        & \cong \sD \Hom_A(- \otimes_{\Gamma_e} (1-e)A, (1-e)A) \\
                                                                        & \cong \sD \Hom_{\Gamma_e}(-, \Hom_A((1-e)A, (1-e)A))
      \end{align*}
      of natural isomorphisms, where the first functor is  $F_e \nu_A G_e$ and the last functor is naturally isomorphic to $\nu_{\Gamma_e}$ as $\Gamma_e \cong \Hom_A((1-e)A, (1-e)A)$.
   \end{proof}

   \begin{proposition} \label{proposition:artranslate}
      There is an isomorphism $\tau_{\Gamma_e}M \cong F_e \tau_A G_e(M)$ for $M \in \mod \Gamma_e$.
   \end{proposition}
   \begin{proof}
         Let $P_1 \to P_0 \to M \to 0$ be an augmented minimal projective presentation.
         \Cref{lemma:sturdy}(\ref{item:projpres}) implies that $G_e(P_1) \to G_e(P_0) \to G_e(M) \to 0$ is an augmented minimal projective presentation of $G_e(M) \in \mod A$.
         Applying $\nu_A$ yields an exact sequence 
         \[0 \to \tau_A G_e(M) \to \nu_A G_e(P_1) \to \nu_A G_e(P_0) \to \nu_A G_e(M)  \to 0\]
         which yields an exact sequence
         \begin{align*} 
            0 \to F_e \tau_A G_e(M) \to \nu_{\Gamma_e }P_1 \to \nu_{\Gamma_e} P_0 \to \nu_{\Gamma_e} M \to 0
         \end{align*}
         upon applying the exact functor $F_e$ and using \Cref{lemma:nakayama} for the three terms on the right.
         But, because $P_1 \to P_0 \to M \to 0$ was an augmented minimal projective presentation we have $F_e \tau_A G_e(M) \cong \tau_{\Gamma_e}M$.
   \end{proof}

   \begin{corollary} \label{corollary:inherittau}
      If $A$ has a \tm and $(1-e)Ae \in \proj \Gamma_e^{\op}$ then $\Gamma_e$ {has a \tm.}
   \end{corollary}
   \begin{proof}
      Let $\Phi_A \colon \sK_0(\mod A) \to \sK_0(\mod A)$ be a \tm of $A$.
      Since $F_e$ and $G_e$ are exact by \Cref{lemma:sturdy}(\ref{item:exact}), they induce linear maps $F_e^\ast \coloneqq \sK_0(F_e) \colon \sK_0(\mod A) \to \sK_0(\mod \Gamma_e)$ and $G_e^\ast \coloneqq \sK_0(G_e) \colon \sK_0(\mod \Gamma_e) \to \sK_0(\mod A)$, respectively.
      \Cref{proposition:artranslate} now implies that $\Phi_{\Gamma_e} \coloneqq F_e^* \Phi_A G_e^*$ is a \tm for $\Gamma_e$. Indeed, for $M \in \ind \mod \Gamma_e$ non-projective
      \[\Phi_{\Gamma_e} [M] = F_e^* \Phi_A G_e^* [M] = F_e^* \Phi_A[G_e(M)] = F_e^* [\tau_A G_e(M)] = [F_e \tau_A G_e (M)] = [\tau_{\Gamma_e} M] \]
      holds, because $G_e(M) \in \ind \mod A$ is non-projective by \Cref{lemma:sturdy}(\ref{item:projpres}).
   \end{proof}

   \begin{corollary} \label{corollary:inherittausummand}
      If $A$ has a \tm and $e \in A$ is a central idempotent then $\Gamma_e$ has a \tm.
   \end{corollary}
   \begin{proof}
      We have $(1-e)Ae = 0$ since $e$ is a central idempotent, so $(1-e)Ae \in \mod \Gamma_e^{\op}$ is projective.
      The result follows from \Cref{corollary:inherittau}.
   \end{proof}

   \begin{remark} \label{remark:inherittausimple}
      If $\sD(Ae)$ is simple then the requirements of \Cref{corollary:inherittau} are satisfied. 
      Indeed, the $A^\op$-module $Ae$ is not a direct summand of $A(1-e)$, as $A^\op$ is basic.
      But, $Ae$ is a simple and projective $A^{\op}$-module.
      Hence, $0 = \Hom_{A^\op}(A(1-e),Ae) \cong (1-e)Ae$.
   \end{remark}
   
   \noindent If $A$ has simple injective modules, then \Cref{remark:inherittausimple} and \Cref{corollary:inherittau} allow us to study a \tm on $A$ by studying a \tm on an idempotent subalgebra.
   The following \namecref{lemma:invtaumap} and \namecref{remark:artranslateduality} are useful as they allow us the same reduction, if $A$ has no injective simple modules but at least one projective simple module.

   \begin{lemma} \label{lemma:invtaumap}
      Suppose $A$ has a \tm $\Phi$. 
      If $A$ has no simple injective modules then $\Phi$ has an inverse $\Phi'$.
      Furthermore, $\Phi'$ is a \tim for $A$.
   \end{lemma}
   \begin{proof}
      Let $\Phi$ be a \tm for $A$.
      Since $\mod A$ has no injective simple modules and since $\sK_0(\mod A)$ is free on the simple $A$-modules, we can define a $\ZZ$-linear morphism $\Phi' \colon \sK_0(\mod A) \to \sK_0(\mod A)$ through $\Phi'[S] = [\tau_A^{-1} S]$ for $S \in \mod A$ simple.
      Further, since $\tau_A^{-1} S$ is not projective and indecomposable for $S \in \mod A$ simple and as $\Phi$ is a \tm, we have $\Phi \Phi'([S]) = [S]$ for $S \in \mod A$ simple.
      This means that $\Phi'$ is a right inverse for $\Phi$, since $\sK_0(\mod A)$ is the free abelian group on the simple $A$-modules. 
      Notice that $\End_{\ZZ}(\sK_0(\mod A))$ is a subring of a matrix ring over $\QQ$ and as a result $\Phi$ and $\Phi'$ are indeed inverse to each other by \cite[Proposition 2.2.1]{serre_matrices_2002}.

      We claim that $\Phi'[M] = [\tau_A^{-1} M]$ holds for any $M \in \ind \mod A$ which is non-injective.
      Indeed if $M \in \ind \mod A$ is non-injective, then $\smash{\tau_A^{-1} M}$ is not projective and indecomposable.
      Therefore, we have $\Phi[\smash{\tau_A^{-1}M}] = [\smash{\tau_A \tau_A^{-1} M}] = [M]$, using that $\Phi$ is a \tm.
      Finally, applying $\Phi'$ to both sides, we obtain that $[\tau_A^{-1} M] = \Phi' [M]$ holds.
   \end{proof} 

   \begin{remark} \label{remark:artranslateduality}
      Notice, we have $\sD \tau_{A^{\op}} \sD(M) \cong \tau_A^{-1}M$ for any $M \in \mod A$. 
      Since the $\kk$-dualities are exact, this implies that $A^{\op}$ has a \tm $\Phi^{\op}$ if and only if $A$ has \tim $\Phi' \coloneqq \sK_0(\sD) \Phi^\op \sK_0(\sD)$.
      Hence, \Cref{lemma:invtaumap} shows that if $A$ has a \tm and no injective simple modules, then $A^{\op}$ has a \tm as well.
   \end{remark}
   
   \noindent Finally, we want to investigate how the Auslander--Reiten translation behaves on modules with a simple socle and a first cosyzygy with decomposable socle.

   \begin{lemma} \label{lemma:badmodule} 
      Suppose $A$ has a $\tau$-map $\Phi$.
      If $M \in \ind \mod A$ satisfies $\Phi [M] \neq [N]$ for every $N \in \mod A$ then any submodule of $M$ is projective. 
   \end{lemma}
   \begin{proof}
      It suffices to show the lemma for all indecomposable submodules $M_0 \subseteq M$.
      First, if $M_0 = M$ then $\Phi [M_0] \neq [\tau_A M_0]$ and hence $M_0$ is projective.

      If $M_0 \subsetneq M$ let $M/M_0 = M_1 \oplus \dots \oplus M_k$ be a decomposition with $M_1, \dots, M_k \in \ind \mod A$. 
      For $1 \leq i \leq k$ the modules $M_i$  cannot be projective as they are proper quotients of the indecomposable module $M$. 
      We have $[M]= [M_0] + \dots + [M_k]$ in $\sK_0(\mod A)$.
      If $M_0$ is also not projective then we have
      \begin{align*} 
         \Phi [M] &= \Phi [M_0] + \cdots + \Phi[M_k] = [\tau_A M_0] + \cdots + [\tau_A M_k] = [\tau_A M_0 \oplus \cdots \oplus \tau_A M_k]
      \end{align*}
      in $\sK_0(\mod A)$ as $\Phi$ is a \tm, showing the lemma by contraposition.
   \end{proof}

   \begin{lemma} \label{lemma:neighbourmodules}
      Let $0 \to M \to I_0 \xrightarrow{f} I_1$ be an augmented minimal injective copresentation of an $A$-module $M$.
      Suppose $I_0$ is indecomposable and $I_1 = I_1' \oplus I_1''$ is a non-trival decomposition yielding a decomposition $f = \smash{[f', f'']^{\trans}}$.
      Then there is an exact sequence 
      \begin{equation}
         0 \to M \to M' \to I_1'' \to \cok f \to \cok f' \to 0 \label{equation:5termsequence}
      \end{equation}
      and an exact sequence 
      \begin{equation}
         0 \to \nu_A^{-1} M \to \nu_A^{-1} M' \to \nu_A^{-1} I_1'' \to \tau_A^{-1} M \to \tau_A^{-1}M' \to 0 \label{equation:5termnakayamasequence}
      \end{equation}
      with $M' = \ker(f')$. 
      Further, $M$ and $M'$ are indecomposable non-injective with $M \not \cong M'$.
   \end{lemma}
   \begin{proof}
   We can draw the undashed arrows of the commutative diagram in \Cref{figure:5termsequence}
   \begin{figure}[ht]
      \includepdf{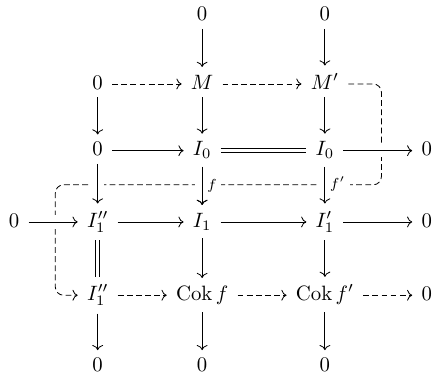}
      \caption{Exact sequences induced by the decomposition $I_1 = I_1' \oplus I_1''$.}
      \label{figure:5termsequence}
   \end{figure}
   where the row $0 \to I_1'' \to I_1 \to I_1' \to 0$ is the split short exact sequence defined through the direct sum decomposition $I_1 = I_1' \oplus I_1''$ and $M'$ is the kernel of $f'$.
   The snake lemma shows that the dashed sequence in \Cref{figure:5termsequence} is exact, which yields the desired sequence (\ref{equation:5termsequence}).

   That $f \colon I_0 \to I_1$ is a minimal injective presentation means that the inclusions $M \to I_0$ and $\im f \to I_1$ are essential extensions.
   Now, the inclusion $M' \to I_0$ is an essential extension as the composite $M \to M' \to I_0$ is an essential extension and the inclusion $\im f' \to I_1'$ is an essential extension as the composite $\im f \to \im f' \oplus \im f'' \to I_1' \oplus I_1''$ is an essential extension.
   Hence, $f' \colon I_0 \to I_1'$ is minimal injective copresentation of $M'$.
   In particular, $M$ and $M'$ are indecomposable as $I_0$ is so and both $M$ and $M'$ are non-injective as they have non-trivial minimal injective copresentations.
   Additionally, $M \not \cong M'$ as $M$ and $M'$ have different minimal injective copresentations.

   We can apply the inverse Nakayama functor to the upper three rows of \Cref{figure:5termsequence} and obtain a commutative diagram in \Cref{figure:5termnakayamasequence}
   \begin{figure}[ht]
      \includepdf{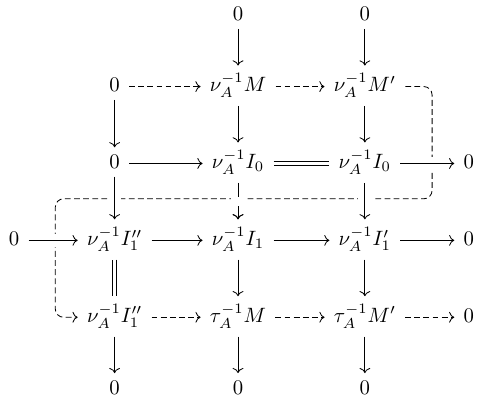}
      \caption{Exact sequences induced by applying $\nu_A^{-1}$ to \Cref{figure:5termsequence}.}
      \label{figure:5termnakayamasequence}
   \end{figure}
   where the two rightmost columns are exact by the definition of $\tau_A^{-1}$ and the two middle rows are exact because functors preserve split exact sequences.
   The dashed sequence in \Cref{figure:5termnakayamasequence} is then exact by the snake lemma.
   This is the sequence (\ref{equation:5termnakayamasequence}).
   \end{proof}

   \begin{lemma} \label{lemma:fork2}
      Suppose that $A$ has a \tm $\Phi$.
      Let $0 \to M \to I_0 \xrightarrow{f} I_1$ be an augmented minimal injective copresentation of an $A$-module $M$.
      If $I_0$ is indecomposable and there is a non-trival decomposition $I_1 = I_1' \oplus I_1''$, where $I_1''$ is indecomposable, then any submodule of $\nu_A^{-1} I_1''$ is projective. 
   \end{lemma}
   \begin{proof}
      \Cref{lemma:neighbourmodules} applied to the decomposition yields a monomorphism $g \colon M \to M'$ coming from (\ref{equation:5termsequence}) and an exact sequence 
      \begin{equation} \label{equation:tauinvseq}
         \nu_{A}^{-1}I_1'' \to \tau_A^{{-1}}M \xrightarrow{\smash{g'}} \tau_A^{{-1}}M' \to 0
      \end{equation}
      coming from (\ref{equation:5termnakayamasequence}).
      Notice, $M$ and $M'$ are indecomposable, non-injective and not isomorphic by \Cref{lemma:neighbourmodules} and hence $\tau_A^{-1}M$ and $\tau_A^{-1}M'$ are indecomposable, non-projective and not isomorphic.
      In particular, $\cok g \neq 0$ and $\ker g' \neq 0$.

      We have $[M] = [M']- [\cok g]$ and $[\tau_A^{{-1}}M] = [\ker g'] + [\tau_A^{{-1}}M']$.
      Applying the $\tau$-map $\Phi$ to the latter equation yields $[M] = \Phi[\ker g'] + [M']$.
      Hence, $\Phi[\ker g'] = -[\cok g]$.
      If there was $N \in \mod A$ such that $\Phi[\ker g'] = [N]$ then $0 = [\cok g \oplus N]$, which is absurd.

      As $\nu_A^{-1}I_1''$ is an indecomposable projective module and $\ker g' \neq 0$, the sequence (\ref{equation:tauinvseq}) gives rise to a projective cover $\nu_A^{-1}I_1'' \to \ker g'$.
      But this implies that $\ker g'$ is indecomposable, as $\nu_A^{-1}I_1''$ is so.
      Hence, any submodule of $\ker g'$ is projective, by \Cref{lemma:badmodule}.
      But then $\ker g'$ is also projective and therefore isomorphic to its projective cover $\nu_A^{-1} I_1''$, showing the lemma as any submodule of $\ker g'$ is projective.
   \end{proof}

   \section{Main results}

   \subsection{Algebras with connected and non-acyclic \texorpdfstring{$\Ext^1$}{Ext1}-quiver admitting a \texorpdfstring{\tm}{τ-map}} \label{subsec:class}

   In this \namecref{subsec:class} we want to prove our main result, a classification of finite-dimensional $\kk$-algebras with connected and non-acyclic $\Ext^1$-quiver admitting a \tm.
   \begin{theorem} \label{theorem:conclusion}
      Suppose the $\Ext^1$-quiver of $A$ is connected and non-acyclic. 
      Then $A$ has a \tm if and only if $A$ is a cyclic Nakayama algebra.
   \end{theorem}

   \noindent The proof of \Cref{theorem:conclusion} is split up.
   In \Cref{theorem:mainthm} we show that all algebras with a connected and non-acyclic $\Ext^1$-quiver admitting a \tm are cyclic Nakayama algebras and in \Cref{proposition:nakayamamap} we show that all Nakayama algebras admit a \tm.
   Using \Cref{remark:quivertoextquiver} we can reformulate this as the following result for quotients of path algebras.

   \begin{corollary} \label{corollary:conclusion}
      Let $\kk$ be an algebraically closed field.
      Suppose $A \coloneqq \kk Q / I$, where $Q$ is a connected and non-acyclic quiver and $I \lhd \kk Q$ is an admissible ideal. 
      Then $A$ has a \tm if and only if $A$ is a cyclic Nakayama algebra.
   \end{corollary}

   \noindent The key to \Cref{theorem:conclusion} lies in \Cref{lemma:noiniter=>nakayama} and \Cref{lemma:case1}.

   \begin{lemma} \label{lemma:noiniter=>nakayama}
      Suppose $A$ admits a \tm $\Phi$.
      If the $\Ext^1$-quiver $Q$ of $A$ has no source and no sink vertices then $A$ is a cyclic Nakayama algebra.
   \end{lemma}
   \begin{proof}
      Since there are no source and sink vertices in $Q$ there are no simple injective and no simple projective modules in $\mod A$, by \Cref{remark:extquiver2}.
      Then $A$ has a \tim $\Phi'$ which is the inverse of $\Phi$ by \Cref{lemma:invtaumap}.
      Suppose $S_1, \dots, S_n$ are the simple $A$-modules.
      Fix the basis $\mathscr{B}$ of $\sK_0(\mod A)$ given by $[S_1], \dots, [S_n]$.
      Then the matrix representation of $\Phi$ with respect to $\mathscr{B}$ has the dimension vectors of $\tau_A S_1, \dots, \tau_A S_n$ as columns and the matrix representation of $\Phi'$ with respect $\mathscr{B}$ has the dimension vectors of $\tau_A^{-1} S_1, \dots, \tau_A^{-1} S_n$ as columns.
      But then, by \Cref{lemma:Nmatrices}, the maps $\Phi$ and $\Phi'$ permute the elements of $\mathscr{B}$.
      Hence, $\tau_A$ permutes the simple $A$-modules and $A$ is a cyclic Nakayama algebra by \Cref{theorem:arsnakayama}.
   \end{proof}

   \begin{lemma} \label{lemma:case1}
      Suppose that
      \begin{enumerate}
         \item the $\Ext^1$-quiver $Q_A$ of $A$ is connected,
         \item there is a source vertex $S' \coloneqq \sD(Ae)$ in $Q_A$ for some idempotent $e \in A$ and
         \item the $\Ext^1$-quiver $Q_{\Gamma_e}$ of $\Gamma_e$ has a component isomorphic to $C_m$ for some $m \in \NN$. 
      \end{enumerate}
      Then $A$ does not admit a \tm.
   \end{lemma}
   \begin{proof}
      Suppose to the contrary that $A$ admits a \tm $\Phi$.
      We may identify $C_m$ with a component of $Q_{\Gamma_e}$ and, by \Cref{lemma:vertexremoval}, we can identify $Q_{\Gamma_e}$ with the full subquiver of $Q_A$ containing all vertices except $S'$.
      As $Q_A$ is connected and as $C_m$ is a component of $Q_{\Gamma_e}$ there must be an arrow $S' \to S$ in $Q_A$ ending at some $S \in C_m$.
      There is a vertex $S'' \in C_m$ and an arrow $S'' \to S$ in $C_m \subseteq Q_A$.

      We have $\Ext_A^1(S', S) \neq 0$ and $\Ext_A^1(S'', S) \neq 0$.
      Let $0 \to S \to I_0 \to I_1$ be an augmented minimal injective copresentation of $S$.
      We have $\Hom_A(S', I_1) \neq 0$ and $\Hom_A(S'', I_1) \neq 0$, by \cite[Lemma 1]{iwanaga} or the dual of \Cref{remark:extwithsimple}.
      Since $S'$ and $S''$ are simple, this implies that $S'$ and $S''$ are submodules of $I_1$.
      Since $I_1$ is injective, there is a non-trivial decomposition $I_1 = I_1' \oplus I_1''$, where $I_1'$ has the injective envelope of $S'$ as a direct summand and $I_1''$ is the injective envelope of $S''$. 
      Notice, $\nu_A^{-1} I_1''$ is indecomposable and hence $\soc(\nu_A^{-1}I_1'')$ is projective by \Cref{lemma:fork2}. 

      But, $\nu_A^{-1}I_1''$ is the projective cover of $S''$. 
      As there is no path in $Q_A$ from $S''$ to any simple $A$-module not in $C_m$, we obtain that all composition factors of $\nu_A^{-1}I_1''$ lie in $C_m$, by \Cref{lemma:extquiver}.
      Hence, all composition factors of $\nu_A^{-1}I_1''$ are non-projective, by \Cref{remark:extquiver2}.
      This includes all the direct summands of $\soc (\nu_A^{-1} I_1'')$ is a contradiction.
   \end{proof}

   \begin{remark} \label{remark:orientednakayama}
      Suppose the $\Ext^1$-quiver $Q$ of a Nakayama algebra $A$ has an oriented cycle and is connected.
      For a fixed simple $A$-module $S$ there exists at most one other simple $A$-module $S'$ with $\Ext_A^1(S,S') \neq 0$, otherwise the projective cover of $S$ would not be uniserial.
      Dually, there exists at most one other simple $A$-module $S''$ with $\Ext_A^1(S'',S) \neq 0$. 
      Hence, there starts and ends at most one arrow in each vertex of $Q$.
      Because $Q$ is connected it must already be isomorphic to its oriented cycle.
   \end{remark}

   \begin{theorem} \label{theorem:mainthm}
      Suppose $A$ admits a \tm.
      If the $\Ext^1$-quiver $Q_A$ of $A$ is connected and has a subquiver isomorphic to $C_m$ for some $m \geq 1$ then $A$ is a cyclic Nakayama algebra.
   \end{theorem}
   \begin{proof}
      We show the theorem by induction on the number $n$ of vertices of $Q_A$.
      For $n = 1$ the $\Ext^1$-quiver $Q_A$ of $A$ can only contain a cycle if $Q_A \cong C_1$.
      Hence, $A$ is a cyclic Nakayama algebra by \Cref{lemma:noiniter=>nakayama}.
      Now, assume that $Q_A$ has $n$ vertices and the theorem holds for all algebras with less than $n$ vertices in its $\Ext^1$-quiver. 
      
      We show that $A$ cannot have any injective simple module. 
      Suppose $A$ has an injective simple module $S' = \sD(Ae)$.
      Then we have $(1-e)Ae = 0$ by \Cref{remark:inherittausimple} and $\Gamma_e$ has a \tm by \Cref{corollary:inherittau}.
      Let $Q_{\Gamma_e}$ be the $\Ext^1$-quiver of $\Gamma_e$. 
      Then $Q_{\Gamma_e}$ is obtained from $Q_A$ by removing the source vertex $S'$ and the arrows incident to it, by \Cref{lemma:vertexremoval}.
      Notice, $Q_{\Gamma_e}$ contains an oriented cycle as $Q_A$ did so.
      Hence, there is a component $C$ of $Q_{\Gamma_e}$ which contains this oriented cycle.
      This component induces an idempotent subalgebra $\Gamma_{C}$ corresponding to a central idempotent, by \Cref{remark:extquiver1}.
      Then $\Gamma_{C}$ admits a \tm, by \Cref{corollary:inherittausummand}.
      Furthermore, $C$ has fewer than $n$ vertices. 
      Hence, $\Gamma_{C}$ must be a Nakayama algebra by induction hypothesis.
      Therefore, $C$ is a component of $Q_{\Gamma_e}$ which is isomorphic to an oriented cycle $C_m$ for some $m \geq 1$, by \Cref{remark:orientednakayama}.
      Applying \Cref{lemma:case1} yields that $A$ does not admit a \tm.
      Therefore, this case cannot occur and $A$ cannot have any injective simple modules.
      
      If $A$ has no injective simple modules but projective simple modules, then $A^{\op}$ has a \tm by \Cref{remark:artranslateduality} and, by $\kk$-duality, the $\Ext^1$-quiver of $A^\op$ is $Q_A^\op$, which has $n$ vertices and an oriented cycle. 
      But $A^\op$ then has injective simple modules, which cannot happen by the above.

      Therefore, $A$ cannot have any injective or projective simple modules.
      But then $A$ is a cyclic Nakayama algebra by \Cref{lemma:noiniter=>nakayama}.
   \end{proof}

   \noindent Finally, we show that every Nakayama algebra admits a \tm.

   \begin{proposition} \label{proposition:nakayamamap}
      Let $A$ be a Nakayama algebra and fix $x_S \in \sK_0(\mod A)$ for each simple projective module $S \in \mod A$.
      Then the unique morphism $\Phi \in \End_\ZZ(\sK_0(\mod A) )$ with
      \begin{equation}\label{eq:taumap}
         \Phi [S] = \begin{cases} 
                        [\tau_A S] & \text{for $S \in \mod A$ simple and non-projective, } \\
                        x_S      & \text{for $S \in \mod A$ simple and projective,} 
                    \end{cases}
      \end{equation}
      is a \tm for $A$. 
      If $A$ has no projective simple modules then $A$ has a unique \tm.
   \end{proposition}
   \begin{proof}
      Every module $M \in \ind \mod A$ is uniserial by \cite[Theorem VI.2.1]{auslander_representation_1995} and hence of the form $P / \rad^l P$ for $P \in \ind \proj A$ and $l \in \NN$ the Loewy length of $M$.
      Let $\Phi$ be the unique $\ZZ$-linear endomorphism of $\sK_0(\mod A)$ defined through (\ref{eq:taumap}).
      We claim by induction on $l \in \NN$ that $[\tau_A M] = \Phi [M]$ for all non-projective $M \in \ind \mod A$ of Loewy length $l$.
      
      For $l = 1$ this follows from the definition of $\Phi$, as an indecomposable module of Loewy length $1$ is simple.
      For the induction step suppose $l \geq 2$.
      We have a short exact sequence
      \begin{equation} \label{equation:arsequence}
         0 \to \rad M \to M \to M / \rad M \to 0.
      \end{equation}
      Write $M \cong P/\rad^l P$ for some $P \in \ind \proj A$.
      Notice, $\smash{\rad^l} P \neq 0$ because $M \cong P /\smash{\rad^l P}$ is non-projective by assumption.
      Therefore, $\rad M \cong \rad P / \smash{\rad^l} P$ is non-projective and uniserial as a non-trivial quotient of the uniserial module $\rad P$.
      Similarly, $M/\rad M$ is non-projective and simple.

      We can now calculate the Auslander--Reiten translates of the modules in (\ref{equation:arsequence}). 
      Because $\rad P / \rad^{l+1} P$ is uniserial of Loewy length $l$ we have $M' \coloneqq \rad P/ \rad^{l+1} P \cong P' / \smash{\rad^{l}} P'$ for some $P' \in \ind \proj A$.
      Applying \cite[Proposition IV.2.6(c)]{auslander_representation_1995} and the third isomorphism theorem for modules yields
      \begin{align*}
         \tau_A(M / \rad M) &\cong \tau_A(P/\rad P) \cong \rad P / \rad^2 P \cong M' /\rad M' \text{,} \\ 
         \tau_A M &\cong \tau_A( P /\rad^l P) \cong \rad P / \rad^{l+1} P \cong M' \text{ and} \\ 
         \tau_A(\rad M) &\cong \tau_A(M'/ \rad^{l-1} M') \cong \tau_A(P'/ \rad^{l-1} P') \cong \rad P' / \rad^{l} P' \cong \rad M'.
      \end{align*}
      But those Auslander--Reiten translates also fit into a short exact sequence
      \[ 0 \to \rad M' \to M' \to M'/ \rad M' \to 0 \]
      which shows
      \begin{align*}
         [\tau_A M] = [M'] = [\rad M'] + [M'/\rad M'] = [\tau_A (\rad M)] + [\tau_A (M/\rad M)]
         \intertext{which is by the induction hypothesis equal to }
               \Phi [\rad M] + \Phi [M/\rad M] = \Phi\big([\rad M] + [M/\rad M]\big) = \Phi[M].
      \end{align*}
      This completes the induction step.

      Finally, if $A$ has no projective simple module, then each \tm $\Phi$ satisfies $\Phi[S] = [\tau_A S]$ for each simple module $S \in \mod A$.
      Hence, the \tm is uniquely determined.
   \end{proof}

   \section{Algebras not covered by this paper}

   There are examples of algebras which admit a $\tau$-map but are neither hereditary nor Nakayama algebras.
   Such algebras can arise from certain kinds of gluings of Nakayama algebras to hereditary algebras.
   For instance the two algebras of the form $\kk Q/\rad^2 \kk Q$ where $Q$ is
   \includepdf{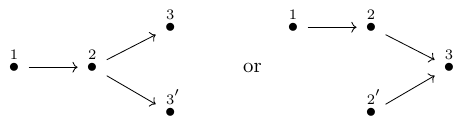}
   both admit a \tm, even though they are neither Nakayama algebras nor hereditary themselves.
   That these algebras indeed admit $\tau$-maps can be read of their Auslander--Reiten quivers which are given by
   \includepdf{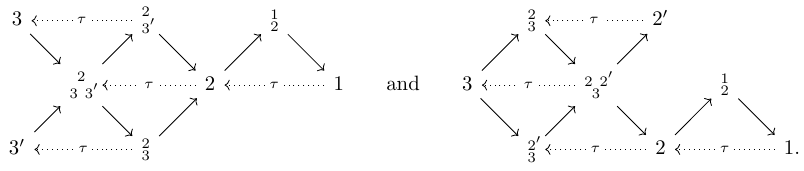}

   Notice, the Auslander--Reiten quiver arises in both cases as the gluing of the Auslander--Reiten quiver of a hereditary algebra and an acyclic Nakayama algebra.

   One can ask the question whether all algebras which admit a \tm and have an acyclic $\Ext^1$-quiver can be constructed by a similar gluing construction.
   We want to put this question for further research.

   \begin{acknowledgment}
      Many thanks to Peter Jørgensen for his patience, corrections and advice.
      Thanks to Steffen König for his suggestions and especially for pointing out \Cref{theorem:arsnakayama} to me.
      Thanks to the anonymous referee for suggestions and corrections. 

      This work has been financially supported by the Aarhus University Research Foundation (grant no. AUFF-F-2020-7-16).
   \end{acknowledgment}
   
   \bibliographystyle{alpha}
   

\end{document}